\documentclass{amsart}

\usepackage{amssymb}
\usepackage{amsthm}
\usepackage{amsfonts}
\usepackage{mathrsfs}
\usepackage{amsmath}
\usepackage{extarrows}

\newtheorem{theorem}{Theorem}[section]
\newtheorem{corollary}[theorem]{Corollary}
\newtheorem{lemma}[theorem]{Lemma}
\newtheorem{proposition}[theorem]{Proposition}

\numberwithin{equation}{section}

\newcommand{\RePt}{\mathrm{Re}\,}
\newcommand{\ImPt}{\mathrm{Im}\,}
\newcommand{\ball}{\mathbb{B}}

\newcommand{\calU}{\mathcal{U}}

\newcommand{\calS}{\mathcal{S}}
\newcommand{\bfi}{\mathbf{i}}

\newcommand{\bfrho}{\boldsymbol{\rho}}

\newcommand{\bbC}{\mathbb{C}}

\begin{document}

\title[Compactness and essential norm]{Compactness and the essential norm on Bergman spaces of the Siegel upper half-space}

\author{Peiyao Li}
\email{peiyaoli@hainanu.edu.cn}
\address{School of Mathematics and Statistics, Hainan University, Haikou, Hainan 570228,
People's Republic of China.}

\author{Congwen Liu}
\email{cwliu@ustc.edu.cn}
\address{CAS Wu Wen-Tsun Key Laboratory of Mathematics,
School of Mathematical Sciences,
University of Science and Technology of China,
Hefei, Anhui 230026,
People's Republic of China}
\thanks{The second author was supported by the National Natural Science Foundation of China (12371084).}

\author{Jiajia Si}
\email{sijiajia@mail.ustc.edu.cn}
\address{School of Mathematics and Statistics, Hainan University, Haikou, Hainan 570228,
People's Republic of China.}
\thanks{The third author was supported by the Hainan Provincial Natural Science Foundation of China (124YXQN413) and the Nanhai New Star Project (2025NHXX203).}


\begin{abstract}
In this paper, we characterize compactness in terms of the Berezin transform and obtain an essential norm estimate for operators in the Toeplitz algebra on Bergman spaces of the Siegel upper half-space.
\end{abstract}


\maketitle

\section{Introduction}

Axler and Zheng \cite{AZ98} proved that a finite sum of finite products of Toeplitz operators with bounded symbols on the Bergman space of the unit disc is compact if and only if its Berezin transform vanishes at the boundary. 
Extensions to other setting appear in \cite{CN09,CSZ18,Eng99}.
On the unit ball, Suárez \cite{Sua07} used a localization property to extend this criterion to the Toeplitz algebra, later generalized to various function spaces \cite{BI12,MSW13,MW14a}. Mitkovski and Wick \cite{MW14} simplified the proof by imposing an integrability condition on translates of the reproducing kernel: if a linear operator $T$ satisfies
\[
\sup_{z\in\ball} \|U_zT k_z\|_{L^p(\ball)} <\infty \quad\text{and}\quad \sup_{z\in\ball} \|U_zT^* k_z\|_{L^p(\ball)} <\infty
\]
for some $p>\frac{4-\kappa}{2-\kappa}$
(here and throughout, $\ball$ denotes the unit ball of $\bbC^n$),
then the localization property follows directly.
This approach has been widely extended \cite{HLW18,HW20,Isr15,IMW15}.

On the Siegel upper half-space $\calU$ of $\bbC^n$, there are two escape regimes (approach to the Euclidean boundary and escape to infinity) preventing the above condition from giving uniform control.
Inspired by  \cite{CN09}, we instead impose the following condition
\[
\sup_{z\in\calU} |T_z K_{\bfi}|\in L_{\alpha/p^\prime}^1(\calU)\quad \text{and} \quad \sup_{z\in\calU} |T_z^* K_{\bfi}|\in L_{\alpha/p}^1(\calU)
\]
for some $-1<\alpha<0$ (see \eqref{eq:essenstial}). It yields the localization estimate (Proposition \ref{prop:localization}) and is satisfied by finite Toeplitz products (Lemma \ref{lem:supS_z}).

Although $\calU$ is biholomorphic to $\ball$, 
the result for $p\ne2$ is not a formal transfer of the ball results cited above.
Let $J_C \Phi$ be the complex Jacobian determinant of the Cayley transform
$\Phi:\ball\to\calU$, and fix a holomorphic branch of $J_\Phi^{2/p}$.  The map
\[
 V_p:A^p(\calU)\longrightarrow A^p(\ball),\qquad
 V_pf=(J_C \Phi)^{2/p}(f\circ\Phi),
\]
is an isometry, however, 
\[
 V_pP_{\calU}V_p^{-1}
 =M_{(J_ C\Phi)^{2/p-1}}P_{\ball}M_{(J_C\Phi)^{1-2/p}}.
\]
Only for $p=2$ do the two multipliers disappear. 
Thus, the Bergman projection and its Toeplitz algebra are not conjugated to their standard ball
counterparts when $p\ne2$.

We now formulate our main result.  The Siegel upper half-space of $\bbC^n$ is
\[
 \calU=\{z\in\bbC^n:\bfrho(z)>0\},
\]
where $\bfrho(z)=\ImPt z_n-|z'|^2$, and we write
\[
 z=(z',z_n),\qquad z'=(z_1,\ldots,z_{n-1})\in\bbC^{n-1},\quad z_n\in\bbC.
\]
For $1<p<\infty$, let $L^p=L^p(\calU,dV)$, with
\[
 \|f\|_p=\left(\int_{\calU} |f(z)|^p \, dV(z)\right)^{1/p},
\]
and let $A^p$ be its closed subspace of holomorphic functions.  The Bergman
projection is
\[
 Pf(z)=\int_{\calU}K(z,w)f(w)\, d V(w),
\]
where
\[
K(z,w)= \frac {n!}{4\pi^{n}}\left[\frac {i}{2} (\overline{w}_{n}-z_{n})
- z^{\prime} \cdot \overline{w^{\prime}} \right]^{-n-1};
\] 
The projection is bounded from
$L^p$ onto $A^p$; see \cite{CR80}.

Set $K_z=K(\cdot,z)$ and $k_z^{(p)}=K_z/\|K_z\|_p$.  
For a bounded linear operator $T$ on $A^p$, its Berezin transform is given by 
\[
\widetilde{T}(z) = \langle Tk_z^{(p)}, k_z^{(p^\prime)}\rangle.
\]
Here and in what follows, $p^{\prime}$ always denotes the conjugate of $p$, i.e., $p^{\prime}=p/(p-1)$.
Thus $|\widetilde{T}(z)|\leq \|T\|$.  The Berezin transform is injective by the
standard polarization argument and the density of the kernel span.

For $u\in L^{\infty}$, let $M_u$ be the multiplication operator by $u$ and set $T_u=PM_u$.
Clearly, $T_u$ is bounded on $A^p$ for $1<p<\infty$.
The Toeplitz algebra $\mathcal{T}_p$ on $A^p$ is the closed subalgebra generated by $\{T_u:u\in L^{\infty} \}$, i.e.,
\[
\mathcal{T}_p=\text{closure} \left\{  \sum_{i=1}^M T_{u_{i1}}\cdots T_{u_{iN_i}}: u_{ij}\in L^\infty  \right\},
\]
where the closure is taken with respect to the operator norm on $A^p$.
Our main result is the following.
\begin{theorem}\label{thm:main}
Let $1<p<\infty$ and $T\in\mathcal{T}_p$. Then 
\begin{itemize}
\item[(i)] $\|T\|_{e, A^p\to A^p} \simeq \sup\limits_{\|f\|_p\leq 1} \limsup\limits_{z\to\partial\widehat{\calU}} \|T_z f\|_p$.
\item[(ii)] $T$ is compact if and only if $\widetilde{T}(z)\to 0$ as $z\to\partial\widehat{\calU}$.
\end{itemize}
\end{theorem}

Here $\widehat{\calU}:=\overline{\calU}\cup\{\infty\}$
and $\partial\widehat{\calU}:=\widehat{\calU}\setminus\calU=\{z:\bfrho(z)=0\}\cup\{\infty\}$.
The normalized translate $T_z$ is defined in Subsection 2.3.

Throughout the paper, we will abbreviate inessential constants involved in inequalities by writing $A\lesssim B$ for positive quantities $A$ and $B$ if the ratio $A/B$ has a positive upper bound. Also, $A\simeq B$ means both $A\lesssim B$ and $B\lesssim A$.

The paper is organized as follows. 
In Section 2, we collect some basic prerequisites.
Section 3 establishes the essential norm characterization (Theorem 3.1) for operators satisfying the integrability condition \eqref{eq:essenstial}.
Finally, Theorem \ref{thm:main} is proved in Section 4.

\section{Preliminaries}

\subsection{Elementary results}
For simplicity, we set
\[
\bfrho(z,w):=\frac{i}{2}(\overline{w}_n-z_n)-z^{\prime} \cdot \overline{w^{\prime}}.
\]
Note that $\bfrho(z)=\bfrho(z,z)$.
With this notation, the Bergman kernel of $\calU$ is given by 
\[
K(z,w) = \frac {n!}{4\pi^n} \frac {1}{\bfrho(z,w)^{n+1}}, \quad z,w\in \calU.
\]

\begin{lemma}[{\cite[Lemma 13]{Liu18}}]\label{lem:FRestimate}
Let $t>-1$ and $s>n+1+t$. The identity
\begin{equation*}\label{eqn:keylem}
\int_{\calU} \frac{\bfrho(w)^{t}} {|\bfrho(z,w)|^{s}} \,dV(w) =
\frac {4 \pi^{n} \Gamma(1+t) \Gamma(s-t-n-1)} {\Gamma^2\left(s/2\right)} \bfrho(z)^{n+1+t-s}
\end{equation*}
holds for all $z\in \calU$.
\end{lemma}

The following Lemmas \ref{lem:elemtryeq1}-\ref{lem:k_zweak} are taken from \cite{LS20}.

\begin{lemma}\label{lem:elemtryeq1}
We have
\begin{equation*}\label{eqn:elemtryeq1}
 2|\bfrho(z,w)|\geq \max\{\bfrho(z),\bfrho(w)\}
\end{equation*}
for any $z, w\in \calU$.
\end{lemma}

\begin{lemma}\label{blem}
Given $r > 0$, the inequalities
\begin{equation*}\label{eqn:eqvltquan}
\frac {1-\tanh (r)}{1+ \tanh (r)} \leq \frac{|\bfrho(z,u)|}{|\bfrho(z,v)|}
\leq \frac {1+\tanh (r)}{1-\tanh (r)}
\end{equation*}
hold for all $z,u,v\in \calU$ with $\beta(u,v)\leq r$, where $\beta(u,v)$ is the Bergman metric on $\calU$ given by 
\begin{equation*}\label{eqn:hyperdist}
 \beta(u,v)=\tanh^{-1}\sqrt{1-\frac{\bfrho(u)\bfrho(v)}{|\bfrho(u,v)|^2}}.
\end{equation*}
\end{lemma}

For $w\in\calU$ and $r>0$, we let
\[
D(w,r)=\{u\in\calU:\beta(u,w)<r\}
\]
denote the Bergman ball centered at $w$ of radius $r$.

\begin{lemma}\label{lem:volBergmanball}
For any $z\in\mathcal{U}$ and $r>0$, we have
\begin{equation*}\label{eqn:volBergmanball}
V(D(z,r)) = \frac {4\pi^n}{n!} \frac {\tanh^{2n} r } {(1-\tanh^2 r)^{n+1}}\, \bfrho(z)^{n+1} .
\end{equation*}
\end{lemma}

\begin{lemma}\label{lem:weak,convergence}
Let $1<p<\infty$ and $\{f_j\}$ be a sequence in $A^p$. Then $f_j\to 0$ weakly in $A^p$ if and only if  $\{f_j\}$ is bounded in $A^p$ and converges to $0$ uniformly on every compact subset of $\calU$.
\end{lemma}

\begin{lemma}\label{lem:k_zweak}
For $1<p<\infty$, $k_z^{(p)}\to 0$ weakly in $A^p$ as $z\to\partial\widehat{\calU}$.
\end{lemma}

We will use the classical Schur's test in its standard $L^p$ form.

\begin{lemma}
Suppose $(X,\mu)$ is a measure space and $R(x,y)$ is a nonnegative measurable function on $X\times X$.
For $1<p<\infty$, if there exist positive constants $C_1$ and $C_2$ and a positive measurable function $h$ on $X$ such that 
\[
\int_{X} R(x,y) h(y)^{p^\prime} \,d\mu(y) \leq C_1 h(x)^{p^\prime}
\]
for almost every $x$ and 
\[
\int_{X} R(x,y) h(x)^{p} \,d\mu(x) \leq C_2 h(y)^p
\]
for almost every $y$, then $Ef(x)=\int_{X} R(x,y) f(y)\, d\mu(y)$ defines a bounded operator on $L^p(X,\mu)$ with $\|E\|\leq C_1^{1/p^\prime} C_2^{1/p}$.
\end{lemma}

\subsection{Möbius transforms}
For each $t>0$, we define the nonisotropic dilation $\delta_t$ by
\[
\delta_t(u)=(t u^{\prime},t^2 u_n), \quad u\in \calU.
\]
For each fixed $z\in\calU$, we associate the following (holomorphic) affine
self-mapping of $\calU$:
\[
h_z(u) := \left(u^{\prime} - z^{\prime}, u_n - \RePt z_n - 2i u^{\prime} \cdot \overline{z^{\prime}}  + i|z^{\prime}|^2 \right),
\quad u\in \calU.
\]
All these mappings are holomorphic automorphisms of $\calU$; see \cite[Chapter XII]{Ste93}. Hence the mappings 
\[
\sigma_z := \delta_{\bfrho(z)^{-1/2}} \circ h_z
\]
are holomorphic automorphisms of $\calU$. 
It is well known that the Bergman metric $\beta$ is invariant under the holomorphic automorphisms.
Therefore, 
\begin{equation}\label{eq:invarianceBeta}
\beta(\sigma_z(u),\sigma_z(v))=\beta(\sigma_z^{-1}(u),\sigma_z^{-1}(v))=\beta(u,v).
\end{equation}
Moreover, $\sigma_z(D(w,r))=D(\sigma_z(w),r)$.
Set $\bfi=(0^{\prime},i)$.
Simple calculations show that $\sigma_z(z)=\bfi$ and 
\begin{equation*}\label{eqn:jacobian}
(J_R \sigma_z) (u) = \bfrho(z)^{-(n+1)},
\end{equation*}
where $(J_R\sigma_z)(u)$ stands for the real Jacobian of $\sigma_z$ at $u$.
Also, directly from the definition, we have
\[
\bfrho(\delta_t(u),\delta_t(v))=t^2 \bfrho(u,v)
\]
and 
\[
\bfrho(h_z(u),h_z(v))=\bfrho(u,v).
\]
Combining the two equalities above yields
\begin{equation}\label{eq:rhosigma_z1}
\bfrho(\sigma_z(u),\sigma_z(v)) = \bfrho(z)^{-1} \bfrho(u,v).
\end{equation}

\subsection{Translation operators}
For each $z\in\calU$, define an adapted translation operator associated with $\sigma_z$ by
\[
U_z^p f = (f\circ \sigma_z) \cdot \bfrho(z)^{-(n+1)/p}.
\]
Note that the $p$ in $U_z^p$ is an index, not a power.
Clearly, each $U_z^p$ is invertible, and its inverse is given by
\[
(U_z^p)^{-1} f = (f\circ \sigma_z^{-1}) \cdot \bfrho(z)^{(n+1)/p}.
\]
By a change of variables, we can see that $\|U_z^p f\|_p=\|f\|_p$ for all $f\in L^p$ and 
$\langle U_z^p f, g\rangle  = \langle  f, (U_z^{p^\prime})^{-1} g\rangle$
for $f\in L^p$ and $g\in L^{p^{\prime}}$. Therefore, $U_z^p$ is an isometric isomorphism on $L^p$ and 
$(U_z^p)^*=(U_z^{p^\prime})^{-1}$.

\begin{lemma}\label{lem:U_z^pK_w}
We have
\[
U_z^p K_w= \bfrho(z)^{(n+1)/p^\prime} K_{\sigma_z^{-1}(w)}
\]
for all $z,w\in\calU$.
\end{lemma}

\begin{proof}
Note from \eqref{eq:rhosigma_z1} that 
\[
K(\sigma_z(u),w) =\bfrho(z)^{n+1} K(u,\sigma_z^{-1}(w)).
\]
Then, a straightforward calculation gives
\begin{align*}
U_z^p K_w(u)&=\bfrho(z)^{-(n+1)/p}K(\sigma_z(u),w)\\
&=\bfrho(z)^{(n+1)/p^\prime} K(u,\sigma_z^{-1}(w))= \bfrho(z)^{(n+1)/p^\prime} K_{\sigma_z^{-1}(w)}(u),
\end{align*}
as desired.
\end{proof}

Note from Lemma \ref{lem:FRestimate} that there is a positive constant $C_p$ independent of $z$ such that 
$\|K_z \|_{p}  = C_p \bfrho(z)^{-(n+1)/p^{\prime}}$. This together with Lemma \ref{lem:U_z^pK_w} easily yields the following identity.

\begin{corollary}\label{cor:U_z^pk_w}
We have
\[
U_z^p k_w^{(p)}= k_{\sigma_z^{-1}(w)}^{(p)}
\]
for all $z,w\in\calU$.
\end{corollary}

For a bounded linear operator $T$ on $L^p$ and  $z\in\calU$, define 
\[
T_z=(U_z^p)^{-1} T U_z^p.
\] 
It is obvious that $\|T_z\|=\|T\|$. Moreover, by Corollary \ref{cor:U_z^pk_w}, we have
\begin{equation}\label{cor:BerzinT}
\widetilde{T}(z)= \langle T_z k_\bfi^{(p)}, k_\bfi^{(p^\prime)}\rangle.
\end{equation}

\begin{lemma}\label{lem:T_varphi_z}
Given $u\in L^\infty$, $(T_{u})_z = T_{u\circ \sigma_z^{-1}}$ for any $z\in \calU$. Consequently, 
if $u_1,u_2,\dots,u_m\in L^{\infty}$, then 
\[
(T_{u_1}T_{u_2}\cdots T_{u_m})_z =T_{u_1\circ\sigma_z^{-1}}T_{u_2\circ\sigma_z^{-1}}\cdots T_{u_m\circ\sigma_z^{-1}}
\]
for any $z\in\calU$.
\end{lemma}

\begin{proof}
For any $f\in A^p$ and $g\in A^{p^\prime}$, by the fact that $(U_z^p)^*=(U_z^{p^\prime})^{-1}$, we have
\begin{align*}
\langle T_{u\circ \sigma_z^{-1}} f, g \rangle
&~=~  \langle (u\circ \sigma_z^{-1}) f, g \rangle ~=~ \langle U_z^p [(u\circ \sigma_z^{-1}) f], U_z^{p^\prime} g \rangle\\
&~=~ \langle u U_z^p f, U_z^{p^\prime} g \rangle ~=~  \langle T_{u} U_z^p f, U_z^{p^\prime} g \rangle \\
&~=~  \langle (U_z^p)^{-1} T_{u} U_z^p f, g \rangle
=\langle (T_{u})_z f,g \rangle.
\end{align*}
Hence, $(T_{u})_z = T_{u\circ \sigma_z^{-1}}$. 
The product formula follows by inserting $U_z^p (U_z^p)^{-1}$ between consecutive factors. 
\end{proof}

\section{A characterization of the essential norm}

The purpose of this section is to prove an essential norm characterization, which plays a key role in the proof of Theorem \ref{thm:main}. 
Recall that the essential norm of a bounded linear operator $T:X\to Y$ is defined by 
\[
\|T\|_{e,X\to Y} = \inf \left\{\|T-C\|_{X\to Y}:  C:X\to Y \, \text{compact}\right\}.
\]

For $\alpha>-1$, let $L_\alpha^1$ denote the space of Lebesgue measurable functions $f$ on $\calU$ such that
\begin{equation*}
\|f\|_{1,\alpha} =\int_{\calU} |f(z)| \bfrho(z)^\alpha \,dV(z)<\infty.
\end{equation*}

\begin{theorem}\label{thm:essenstial}
Let $1<p<\infty$ and let $T$ be a bounded linear operator on $A^p$. 
Assume that, for some $-1<\alpha<0$,
\begin{equation}\label{eq:essenstial}
\sup_{z\in\calU} |T_z K_{\bfi}|\in L_{\alpha/p^\prime}^1\quad \text{and} \quad \sup_{z\in\calU} |T_z^* K_{\bfi}|\in L_{\alpha/p}^1.
\end{equation}
Then we have
\[
\|T\|_{e, A^p\to A^p} \simeq \sup_{\|f\|_p\leq 1} \limsup_{z\to\partial\widehat{\calU}} \|T_z f\|_p.
\]
The constants depend only on $n$ and $p$.
\end{theorem}

Here and in what follows, the notation $T_z^*$ is unambiguous since $(T_z)^* = (T^*)_z$.
So $T_z^*$ can be interpreted in either way whenever it appears.
We shall adapt the method of Mitkovski and Wick \cite[Theorem 4.3]{MW14}.

\begin{lemma}\label{lem:sigma_ztoinfty}
For every $w\in\calU$, we have  $\sigma_z^{-1}(w)\to\partial\widehat{\calU}$ as $z\to\partial\widehat{\calU}$.
\end{lemma}

\begin{proof}
We first claim that
$z\to \partial\widehat{\calU}$ if and only if $\beta(z,\bfi)\to\infty$.
Indeed, observe that $z\to \partial\widehat{\calU}$ if and only if $|\Phi^{-1}(z)|\to 1$
if and only if $\beta_\ball(\Phi^{-1}(z),0)\to \infty$, where $\Phi:\ball\to\calU$ is the Cayley transform and $\beta_\ball$ is the Bergman metric of $\ball$.
Recall from  \cite[Proposition 1.4.15]{Kra01} that 
\[
\beta_\ball(\Phi^{-1}(z),0)=\beta(z,\bfi),
\]
proving the claim. 

%Indeed, observe that 
%\begin{equation*}\label{eqn:rho(z,i)}
%\sqrt{|z|^2+1}\leq2|\bfrho(z,\bfi)|\leq |z|+1\quad \text{and} \quad \bfrho(z)\leq |z|.
%\end{equation*}
%Then we have
%\[
%4\bfrho(z)/(|z|+1)^2 \leq \bfrho(z)/|\bfrho(z,\bfi)|^2 \leq 4\min\{\bfrho(z), |z|/(|z|^2+1) \}.
%\]
%This implies that $z\to \partial\widehat{\calU}$ if and only if $\bfrho(z)/|\bfrho(z,\bfi)|^2\to 0$.
%Recall that $\tanh\beta(z,\bfi)=(1-\bfrho(z)/|\bfrho(z,\bfi)|^2)^{1/2}$, proving the claim. 

By the claim, it now suffices to show that $\beta(\sigma_z^{-1}(w),\bfi)\to\infty$ as $z\to\partial\widehat{\calU}$.  
Using the invariance of the Bergman metric \eqref{eq:invarianceBeta}, we have
\[
\beta(\sigma_z^{-1}(w),\bfi) \ge \beta(z,\bfi) - \beta(\sigma_z^{-1}(w),z)
= \beta(z,\bfi) - \beta(w,\bfi).
\]
Since $\beta(z,\bfi)\to\infty$ while $\beta(w,\bfi)$ is fixed, the right-hand side tends to infinity, completing the proof.
\end{proof}

\begin{lemma}\label{lem:essenstial1}
For any compact operator $C$ on $A^p$ and any $f\in A^p$, we have that $\|C_zf\|_p\to 0$ as $z\to\partial\widehat{\calU}$.
\end{lemma}

\begin{proof}
Since $span\{k_w^{(p)}:w\in\calU\}$ is dense in $A^p$, it suffices to prove the statement for $f=k_w^{(p)}$ with arbitrary $w\in\calU$. 
It follows from Corollary \ref{cor:U_z^pk_w} that 
\[
\|C_z k_w^{(p)}\|_p= \|(U_z^p)^{-1} C k_{\sigma_z^{-1}(w)}^{(p)}\|_p.
\]
This together with Lemmas \ref{lem:sigma_ztoinfty} and \ref{lem:k_zweak} gives the proof.
\end{proof}

\begin{lemma}\label{lem:T_G}
For every $\beta$-bounded subset $G$ of  $\calU$, the operator $PM_{\chi_G}$ is compact on $A^p$.
\end{lemma}

\begin{proof}
Let $\{f_j\}\subset A^p$ be a sequence converging weakly to zero. 
Such a sequence must be bounded, say, $\|f_j\|_p\leq C_1$ for all $j$.
This implies by \cite[Corollary 2.10]{LS20} that $\{f_j\}$ is uniformly bounded on each $\beta$-bounded subset of $\calU$, namely, $|f_j(z)|\leq C_2$ for all $z\in G$ and all $j$.
On the other hand, $f_j\xrightarrow{w} 0$ implies that $f_j\to 0$ pointwise by the reproducing property of the kernel $K_z\in A^{p^\prime}$. Therefore, we may apply the Lebesgue dominated theorem to conclude that 
\[
\| PM_{\chi_G} f_j\|_p \leq \|P\|_{L^p\to A^p} \left(\int_{G} |f_j(z)|^p dV(z)\right)^{1/p} \to 0
\]
as $j\to\infty$.
It follows that $PM_{\chi_G}$ maps weakly convergent sequences into norm convergent ones, and so is compact.
\end{proof}

\begin{lemma}\label{lem:covering}
There exists a positive integer $N=N(n)$ such that for any $r>0$ one can find a locally finite Borel partition $\mathcal{F}_r=\{F_j\}$ of $\calU$ satisfying the following properties:
\begin{itemize}
\item[(i)] each point belongs to at most $N$ of the sets $G_j:=\{z\in\calU: \beta(z,F_j)< r\}$.
\item[(ii)] $\sup_j \text{diam}_\beta F_j \leq C_r<\infty$.
\end{itemize}
The partition can be enumerated so that $\beta(\bfi,F_j)\to\infty$.
\end{lemma}

\begin{proof}
Up to normalization, $(\calU,\beta)$ is complex hyperbolic $n$-space.  It is a
Gromov-hyperbolic geodesic space with bounded growth at some scale, so
Roe's metric-space theorem gives finite asymptotic dimension \cite{Roe2005};
see also Section~2.5 of the corrected arXiv version of
\cite{MW14}.  At scale $2r$, this
gives an integer $N=N(n)$ and a uniformly bounded, locally finite cover $\{E_j\}$ such that
every ball of radius $2r$ meets at most $N$ members of the cover.  After
enumerating the cover, set
\[
 F_1=E_1,\qquad F_j=E_j\setminus\bigcup_{k<j}E_k,
\]
and discard empty sets.  Then $\{F_j\}$ is a Borel partition and
$\text{diam}_\beta F_j\leq C_r:=\sup_j\text{diam}_\beta E_j$.
If $z\in G_j$, then $D(z,r)$ meets $F_j\subset E_j$, so (i) follows.  If
$G_j$ meets $D(z,r)$, then $E_j$ meets $D(z,2r)$; hence $\{G_j\}$ is locally
finite as well.  Properness now permits an enumeration for which
$\beta(\bfi,F_j)\to\infty$.
\end{proof}

The proof of Theorem \ref{thm:essenstial} relies on the following localization property.

\begin{proposition}\label{prop:localization}
Let $1<p<\infty$ and let $T$ be a bounded linear operator on $A^p$. 
Assume that, for some $-1<\alpha<0$,
\[
\sup_{z\in\calU} |T_z K_{\bfi}|\in L_{\alpha/p^\prime}^1 \quad \text{and} \quad \sup_{z\in\calU} |T_z^* K_{\bfi}|\in L_{\alpha/p}^1.
\]
Then for any $\varepsilon>0$ 
there exists $r>0$ such that for the locally finite Borel partition $\mathcal{F}_r=\{F_j\}$ of $\calU$ (from Lemma \ref{lem:covering}), we have
\[
\|T-\sum_j M_{\chi_{F_j}} TP M_{\chi_{G_j}}\|_{A^p\to L^p} <\varepsilon.
\]
\end{proposition}

\begin{proof}
Clearly,  the kernel function of $T-\sum_j M_{\chi_{F_j}} TP M_{\chi_{G_j}}$ on $A^p$ is given by 
\[
\sum_j \chi_{F_j}(z) \chi_{G_j^c}(w) \overline{\langle T^* K_z, K_w \rangle}.
\]
Set 
\[
R_r(z,w)=\sum_j \chi_{F_j}(z) \chi_{G_j^c}(w) |\langle T^* K_z, K_w \rangle|
\] 
and the corresponding operator
\[
E_r f(z)=\int_{\calU} R_r(z,w) f(w)\, dV(w), \quad f\in A^p .
\]

For a fixed $z\in\calU$, since $\{F_j\}$ is a partition of $\calU$, there is a unique $j_0$ such that $z\in F_{j_0}$ and consequently $D(z,r)\subset G_{j_0}$. 
Then, by Lemma \ref{lem:U_z^pK_w} we have
\begin{align*}
\int_{\calU}& R_r(z,w) \bfrho(w)^{\alpha/p} \,dV(w)=\int_{G_{j_0}^c}  |\langle T^* K_z, K_w \rangle| \bfrho(w)^{\alpha/p} \,dV(w)\\
&\leq\int_{D(z,r)^c} |\langle T^* K_z, K_w \rangle| \bfrho(w)^{\alpha/p} \,dV(w)\\
&=\int_{D(\bfi,r)^c} |\langle T^* K_z, K_{\sigma_z^{-1}(w)} \rangle| \bfrho(\sigma_z^{-1}(w))^{\alpha/p} \bfrho(z)^{n+1}\,dV(w)\\
&=\bfrho(z)^{\alpha/p} \int_{D(\bfi,r)^c}  |\langle T^* U_z^p K_\bfi, U_z^{p^\prime}K_w \rangle| \bfrho(w)^{\alpha/p}  \,dV(w)\\
&=\bfrho(z)^{\alpha/p} \int_{D(\bfi,r)^c} |T_z^* K_{\bfi}(w)| \bfrho(w)^{\alpha/p}  \,dV(w):=C_1(r,z) \bfrho(z)^{\alpha/p}.
\end{align*}
For a fixed $w\in\calU$, let $J$ be a subset of all indices $j$ such that $w\notin G_j$. 
Then $\cup_{j\in J} F_j \subset D(w,r)^c$. Thus,   
\begin{align*}
\int_{\calU}  R_r(z,w) \bfrho(z)^{\alpha/p^\prime} \,dV(z)&=\int_{\cup_{j\in J} F_j } |\langle T^* K_z, K_w \rangle| \bfrho(z)^{\alpha/p^\prime} \,dV(z)\\
&\leq \int_{D(w,r)^c} |\langle T^* K_z, K_w \rangle| \bfrho(z)^{\alpha/p^\prime} \,dV(z)\\
&=\int_{D(w,r)^c} |\langle T K_w, K_z \rangle| \bfrho(z)^{\alpha/p^\prime} \,dV(z).
\end{align*}
Applying the same estimates as above, the last integral equals
\[
\bfrho(w)^{\alpha/p^\prime} \int_{D(\bfi,r)^c} |T_w K_{\bfi}(z)| \bfrho(z)^{\alpha/p^\prime}  \,dV(z):=C_2(r,w) \bfrho(w)^{\alpha/p^\prime}.
\]

For any $\varepsilon>0$, the hypothesis on $T$ implies that $C_1(r,z), C_2(r,w)<\varepsilon$ uniformly for all $z,w\in\calU$ provided $r$ is sufficiently large.
Applying Schur's test to the operator $E_r$ with testing function $h(z)=\bfrho(z)^{\alpha/(pp^\prime)}$,
it follows that $E_r$ is bounded from $A^p$ to $L^p$ with $\|E_r\|_{A^p \to L^p}<\varepsilon$. 
This completes the proof of the proposition.
\end{proof}

\begin{proof}[Proof of Theorem \ref{thm:essenstial}]
For any compact operator $C$ on $A^p$ and any $f\in A^p$ with $\|f\|_p\leq 1$, by Lemma \ref{lem:essenstial1} we have
\[
\limsup_{z\to\partial\widehat{\calU}} \|T_z f\|_p
\leq \limsup_{z\to\partial\widehat{\calU}} \|(T-C)_z f\|_p+\limsup_{z\to\partial\widehat{\calU}} \|C_z f\|_p
\leq \|T-C\|.
\]
It immediately follows that 
\[
\sup_{\|f\|_p\leq 1} \limsup_{z\to\partial\widehat{\calU}} \|T_z f\|_p \leq \|T\|_{e, A^p \to A^p}.
\]

For the reverse estimate, observe first that
\[
 \|T\|_{e,A^p\to L^p}
 \leq \|T\|_{e,A^p \to A^p}
 \leq \|P\|_{L^p\to A^p} \|T\|_{e,A^p\to L^p}.
\]
The first inequality follows from the isometric inclusion $A^p\hookrightarrow
L^p$.  For the second, if $C^{\prime}:A^p\to L^p$ is compact, then $PC^{\prime}:A^p\to A^p$
is compact and, since $PT=T$,
\[
 \|T-PC^{\prime}\|_{A^p\to A^p}
 \leq \|P\|_{L^p\to A^p} \|T-C^{\prime}\|_{A^p\to L^p}.
\]
Therefore, it suffices to estimate $\|T\|_{e,A^p\to L^p}$. 

Given $\varepsilon>0$, choose $r$ and the partition from Proposition \ref{prop:localization} . Then we have
\[
\|T-\sum_j M_{\chi_{F_j}} TP M_{\chi_{G_j}}\|_{A^p\to L^p} <\varepsilon.
\]
Note from Lemma \ref{lem:T_G} that the finite sum  $\sum_{j< m} M_{\chi_{F_j}} TP M_{\chi_{G_j}}$ is compact for every $m\in\mathbb{N}$. Hence,
\begin{align*}
\|T\|_{e,A^p\to L^p} &\leq \|TP-\sum_{j< m} M_{\chi_{F_j}} TP M_{\chi_{G_j}}\|_{A^p\to L^p}\\
&\leq \|TP-\sum_j M_{\chi_{F_j}} TP M_{\chi_{G_j}}\|_{A^p\to L^p} +\|T_m\|_{A^p\to L^p}\\
&< \varepsilon +\|T_m\|_{A^p\to L^p},
\end{align*}
where 
\[
T_m= \sum_{j\geq m}  M_{\chi_{F_j}} TP M_{\chi_{G_j}}.
\]
Thus, the proof reduces to showing that
\[
\limsup_{m\to\infty} \|T_m\|_{A^p\to L^p} \lesssim \sup_{\|f\|_p\leq 1} \limsup_{z\to\partial\widehat{\calU}} \|T_z f\|_p .
\]

Take $f\in A^p$ with $\|f\|_p \leq 1$. Since $\{F_j\}$ are disjoint, we have
\begin{align*}
\|T_m f\|_p^p &= \sum_{j\geq m} \| M_{\chi_{F_j}} TP M_{\chi_{G_j}} f\|_p^p\\
&= \sum_{j\geq m} \left\|M_{\chi_{F_j}} T\left( \frac{P M_{\chi_{G_j}}f}{\|M_{\chi_{G_j}} f\|_p}\right) \right\|_p^p \|M_{\chi_{G_j}} f\|_p^p \leq N \sup_{j\geq m} \|T l_j\|_p^p,
\end{align*}
where $N$ comes from Lemma \ref{lem:covering} and 
\[
l_j= \frac{P M_{\chi_{G_j}}f}{\|M_{\chi_{G_j}} f\|_p}.
\]
Therefore, 
\[
\limsup_{m\to\infty} \|T_m\|_{A^p\to L^p} \leq N^{1/p}
\limsup_{j\to\infty} \sup_{\|f\|_p\leq 1} \left\{ \|Tl_j\|_p: l_j= \frac{P M_{\chi_{G_j}}f}{\|M_{\chi_{G_j}} f\|_p} \right\}.
\]
For convenience, let $d\lambda(w)=\bfrho(w)^{-(n+1)}dV(w)$. It is a Möbius invariant measure on $\calU$. 
For each $j$, choose a function $f_j$ in $A^p$ with $\|f_j\|_p\leq 1$ such that
\[
 \sup_{\|f\|_p\leq 1} \left\{ \|T l_j\|_p: l_j= \frac{P M_{\chi_{G_j}}f}{\|M_{\chi_{G_j}} f\|_p} \right\} 
 \leq \|T g_j\|_p+\varepsilon,
\]
where
\begin{align*}
g_j= \frac{P M_{\chi_{G_j}}f_j}{\|M_{\chi_{G_j}} f_j\|_p}
=\frac{ C_p\int_{G_j} \langle f_j, k_w^{(p^\prime)}\rangle k_w^{(p)} \,d\lambda(w)}{\left(\int_{G_j} |\langle f_j, k_w^{(p^\prime)}\rangle|^p \,d\lambda(w)\right)^{1/p}}
= \int_{G_j} a_j(w) k_w^{(p)} \, d\lambda(w),
\end{align*}
with
\[
a_j(w)=\frac{C_p \langle f_j, k_w^{(p^\prime)}\rangle}{\left(\int_{G_j} |\langle f_j, k_w^{(p^\prime)}\rangle|^p \,d\lambda(w)\right)^{1/p}}.
\]
Hence,
\[
\limsup_{m\to\infty} \|T_m\|_{A^p\to L^p} \leq N^{1/p} \left(\limsup_{j\to\infty}  \|T g_j\|_p+\varepsilon\right).
\]

For each $j$, pick $z_j \in G_j$.  It is obvious that $z_j \to \partial\widehat{\calU}$.
By a change of variables $w\mapsto \sigma_{z_j}^{-1}(w)$, we have
\[
g_j= \int_{\sigma_{z_j}(G_j)} a_j(\sigma_{z_j}^{-1}(w)) U_{z_j}^p k_w^{(p)} \, d\lambda(w).
\]
Define
\[
h_j=  \int_{\sigma_{z_j}(G_j)} a_j(\sigma_{z_j}^{-1}(w)) k_w^{(p)} \,d\lambda(w).
\]
Observe that $G_j\subset D(z_j,l_r)$with $l_r=C_r+3r$, and consequently $\sigma_{z_j}(G_j)\subset D(\bfi,l_r)$.
Applying Hölder's inequality together with Lemmas \ref{blem} and \ref{lem:volBergmanball}, we obtain
\begin{align*}
|h_{j}(z)| &\leq \left(\int_{\sigma_{z_j}(G_j)} |a_j(\sigma_{z_j}^{-1}(w))|^p \,d\lambda(w) \right)^{1/p}
\left(\int_{\sigma_{z_j}(G_j)} |k_w^{(p)}(z)|^{p^\prime} \,d\lambda(w) \right)^{1/p^\prime}\\
&\lesssim \left( \int_{G_j} |a_j(w)|^p \,d\lambda (w)\right)^{1/p}
\left( \int_{D(\bfi,l_r)} | \bfrho(z,w)|^{-(n+1)p^\prime} \,dV(w)\right)^{1/p^\prime}\\
&\lesssim |\bfrho(z,\bfi)|^{-(n+1)} V(D(\bfi,l_r))^{1/p^\prime} \lesssim |\bfrho(z,\bfi)|^{-(n+1)}.
\end{align*} 
Here the constants depend only on $r$.
By Lemma \ref{lem:FRestimate}, $|\bfrho(\cdot, \bfi)|^{-p(n+1)}$ is integrable.

We claim that $g_j=U_{z_j}^p h_j$. Since $(U_{z_j}^p)^*=(U_{z_j}^{p^\prime})^{-1}$, it suffices to verify that 
$\langle g_j,b\rangle=\langle h_j, (U_{z_j}^{p^\prime})^{-1}b \rangle$ for every $b\in A^{p^\prime}$.
This can be achieved by a straightforward calculation.
Consequently,
\[
\|h_j\|_p=\|g_j\|_p \leq \|P\|_{L^p \to A^p},
\]
and
\[
 \limsup_{j\to\infty} \|T g_j\|_p
=\limsup_{j\to\infty} \|TU_{z_j}^p h_j\|_p =\limsup_{j\to\infty} \|T_{z_j}h_j\|_p.
\]

Now pick a subsequence $\{h_{j_k}\}$ such that 
\[
\limsup_{j\to\infty} \|T_{z_j} h_{j}\|_p\leq \lim_{k\to\infty} \|T_{z_{j_k}} h_{j_k}\|_p+\varepsilon.
\]
Since $\{h_{j_k}\}$ is bounded in $A^p$, the Banach-Alaoglu theorem implies that, after passing to a further subsequence if necessary,  it converges weakly to some $h\in A^p$. 
By Lemma \ref{lem:weak,convergence}, $h_{j_k}\to h$ pointwise. Then the dominated convergence theorem yields that $h_{j_k}\to h$ in $A^p$.
In particular, $\|h\|_p\leq \|P\|_{L^p \to A^p}$.
Therefore, we have
\begin{align*}
&\lim_{k\to\infty} \|T_{z_{j_k}} h_{{j_k}}\|_p
\leq \limsup_{k\to\infty}  \|T_{z_{j_k}} (h_{{j_k}}-h)\|_p+ \limsup_{k\to\infty} \|T_{z_{j_k}}h\|_p\\
&\leq \lim_{k\to\infty} \|T\| \,\|(h_{{j_k}}-h)\|_p+ \limsup_{k\to\infty} \|T_{z_{j_k}}h\|_p
\leq \|P\|_{L^p\to A^p}  \sup_{\|f\|_p\leq 1} \limsup_{z\to\partial\widehat{\calU}} \|T_z f\|_p.
\end{align*}

Putting everything together, we obtain
\[
\limsup_{m\to\infty} \|T_m\|_{A^p\to L^p}
\leq N^{1/p} \left(\|P\|_{L^p\to A^p} \sup_{\|f\|_p\le 1} \limsup_{z\to\partial\widehat{\calU}} \|T_z f\|_p + \varepsilon\right).
\]
Since $\varepsilon$ is arbitrary, the desired inequality follows. The proof of the theorem is complete.
\end{proof}

\section{Proof of Theorem \ref{thm:main}}

In the final section we prove Theorem \ref{thm:main}.
Another key ingredient is Lemma \ref{lem:supS_z}.
We first record the logarithmic integral estimate needed below.
For $a>-1$, $c\geq 0$, put
\[
J_{a,b,c}(\eta)=\int_{\ball} \frac{(1-|\xi|^2)^a l(\xi)^c}{|1-\eta\cdot\overline{\xi}|^{n+1+a+b}} \,dV(\xi), 
\]
where
\[
l(\xi)=\log\frac{e}{1-|\xi|^2}, \quad \xi\in\ball.
\]
It follows from \cite[Theorem 3.1]{ZLSG18} that $J_{a,b,c}(\eta)$ is bounded for $\eta\in\overline{\ball}$ if $b<0$,
and 
\[
J_{a,0,c}(\eta)\simeq l(\eta)^{c+1}, \quad \text{for}\,\, \eta\in\ball,
\]
where the constants depend on $a,b,c$ and $n$.

Recall that the Cayley transform $\Phi:\ball\to \calU$ is given by
\[
(z^{\prime}, z_{n})\; \longmapsto\; \left( \frac {z^{\prime}}{1+z_{n}},
i\left(\frac {1-z_{n}}{1+z_{n}}\right) \right).
\]
A straightforward calculation shows that
\[
\bfrho(\Phi(\xi),\Phi(\eta)) = \frac {1-\xi\cdot \overline{\eta}} {(1+\xi_{n}) (1+\overline{\eta}_{n})}
\]
for all $\xi,\eta\in \ball$, and
the real Jacobian of $\Phi$ at $\xi\in \ball$ is
\[
\left(J_{R}\Phi\right)(\xi) = \frac {4}{|1+\xi_{n}|^{2(n+1)}}.
\]
We refer to \cite[Chapter XII]{Ste93} for properties of the Cayley transform.

Denote
\[
L(z)=1+\log\frac{|\bfrho(z,\bfi)|^2}{\bfrho(z)}.
\]
It is obvious that 
\[
l(\xi)=L(\Phi(\xi)).
\]

\begin{lemma}\label{lem:forsupS_z}
For any $c\geq 0$ and any $w\in\calU$, we have
\begin{align*}
 \int_{\calU} &\frac{1}{|\bfrho(w,v)|^{n+1} |\bfrho(v,\bfi)|^{n+1}} L(v)^c \,dV(v) \simeq \frac{1}{|\bfrho(w,\bfi)|^{(n+1)}}  L(w)^{c+1}.
 \end{align*}
The constants depend only on $n$ and $c$.
\end{lemma}

\begin{proof}
Write $w=\Phi(\eta)$ and make a change of variables $v=\Phi(\xi)$. Then the  integral becomes
\[
4 |1+\eta_n|^{n+1} J_{0,0,c}(\eta) \simeq |1+\eta_n|^{n+1} l(\eta)^{c+1}
=\frac{1}{|\bfrho(w,\bfi)|^{(n+1)}} L(w)^{c+1},
\]
finishing the proof.
\end{proof}

For a real number $t$, let $\mathcal{S}_{t}$ denote the space of all holomorphic functions $f$ on $\calU$ such that
\[
\|f\|_{\calS_t}:=\sup_{z\in\calU} |\bfrho(z,\bfi)|^{t} |f(z)| < \infty.
\]
It is known from \cite[Theorem 4.1]{LS20} that $\mathcal{S}_{n+1}$ is dense in $A^p$ for every $1<p<\infty$.

\begin{lemma}\label{lem:supS_z}
Let $1<p<\infty$ and let $S$ be a finite sum of finite products of Toeplitz operators with bounded symbols. For every $f\in \mathcal{S}_{n+1}$, we have
\begin{itemize} 
\item[(a)]  $\sup\limits_{z\in\mathcal{U}} |S_z f|\in L^p$.
 \item[(b)] $\sup\limits_{z\in\mathcal{U}} |S_z f|\in L_{\alpha/p^\prime}^1$ and $\sup\limits_{z\in\mathcal{U}} |S_z^* f|\in L_{\alpha/p}^1$ for every $-1<\alpha<0$. 
 \end{itemize}
\end{lemma}

\begin{proof}
Fix $f\in \mathcal{S}_{n+1}$. 
By linearity, it suffices to consider a single product
\[
S=T_{u_1}T_{u_2}\cdots T_{u_m}, \quad u_1,u_2,\dots,u_m\in L^{\infty}.
\]
By Lemma \ref{lem:T_varphi_z}, for every $z\in\calU$, we have
\[
S_z=T_{u_1\circ\sigma_z^{-1}}T_{u_2\circ\sigma_z^{-1}}\cdots T_{u_m\circ\sigma_z^{-1}}.
\]

We claim that, for each $w\in\calU$, 
\begin{equation}\label{eq:supS_z}
\sup_{z\in\calU} |S_zf(w)|\lesssim  \left(\prod_{j=1}^m \|u_j\|_{\infty}\right) \|f\|_{\calS_{n+1}} |\bfrho(w,\bfi)|^{-(n+1)}  L(w)^m.
\end{equation}
Once this is proved, set
\[
G_m(w)= |\bfrho(w,\bfi)|^{-(n+1)}  L(w)^m.
\]
Then (a) and (b) follow if we can show that $G_m$ belongs to the required spaces.
By the argument at the beginning of the section, we deduce that, for $1<p<\infty$,
\begin{align*}
\|G_m\|_p^p &=\int_{\calU} \frac{1}{|\bfrho(w,\bfi)|^{p(n+1)}}  L(w)^{pm} \,dV(w)\\
&=4\int_{\ball} \frac{1}{|1+\eta_n|^{(2-p)(n+1)}} l(\eta)^{pm} \,dV(\eta)<\infty,
\end{align*}
and for $-1<\alpha<0$,
\[
\|G_m\|_{1,\alpha/p^\prime} = 4  \int_{\ball} \frac{(1-|\eta|^2)^{\alpha/p^\prime}}{|1+\eta_n|^{n+1+2(\alpha/p^\prime)}} l(\eta)^{m} \,dV(\eta)<\infty.
\]
The same estimate holds for $G_m$ in $L_{\alpha/p}^1$.
For $S_z^*$, simply note that $(T_{u_1}T_{u_2}\cdots T_{u_m})^*=T_{\overline{u_m}}\cdots T_{\overline{u_1}}$, which is straightforward to verify. 

It remains to prove the claim. We proceed by induction.
For $m=1$, Lemma \ref{lem:forsupS_z} gives, for any $z\in\calU$,
\begin{align*}
|T_{u_1\circ\sigma_z^{-1}} f(w)| &\lesssim \|u_1\|_{\infty}\int_{\calU} \frac{|f(v)|}{|\bfrho(w,v)|^{n+1}} \,dV(v)\\
&\leq \|u_1\|_{\infty}  \|f\|_{\calS_{n+1}} \int_{\calU} \frac{dV(v)}{|\bfrho(w,v)|^{n+1} |\bfrho(v,\bfi)|^{n+1}}\\
&\simeq \|u_1\|_{\infty}   \|f\|_{\calS_{n+1}}  |\bfrho(w,\bfi)|^{-(n+1)} L(w).
\end{align*}
Thus \eqref{eq:supS_z} holds for $m=1$. 
Assume that \eqref{eq:supS_z} holds for any product of length $m-1$. Write
$R_z= T_{u_2\circ\sigma_z^{-1}}\cdots T_{u_m\circ\sigma_z^{-1}}$.
Then for any $z\in\calU$,
\begin{align*}
&|S_z f(w)|= |T_{u_1\circ\sigma_z^{-1}} R_z f(w)| \lesssim \|u_1\|_{\infty}\int_{\calU} \frac{|R_zf(v)|}{|\bfrho(w,v)|^{n+1}} \,dV(v)\\
&\lesssim \left(\prod_{j=1}^m \|u_j\|_{\infty}\right) \|f\|_{\calS_{n+1}}
\int_{\calU} \frac{1}{ |\bfrho(w,v)|^{(n+1)} |\bfrho(v,\bfi)|^{n+1}}  L(v)^{m-1}dV(v)\\
&\lesssim \left(\prod_{j=1}^m \|u_j\|_{\infty}\right) \|f\|_{\calS_{n+1}} |\bfrho(w,\bfi)|^{-(n+1)}  L(w)^m,
\end{align*}
where the last inequality uses again Lemma \ref{lem:forsupS_z}.
This completes the proof.
\end{proof}

Let WOT denote the weak operator topology in the space of bounded linear operators on $A^p$.

\begin{lemma}\label{lem:B(T)inC_0}
Let $1<p<\infty$ and $T$ be a bounded linear operator on $A^p$. 
As $z\to\partial\widehat{\calU}$, we have $\widetilde{T}(z)\to 0$ if and only if $T_z\xrightarrow{\text{WOT}} 0$.
\end{lemma}

\begin{proof}
The sufficiency is immediate from \eqref{cor:BerzinT}. It suffices to prove the necessity. Suppose that $\widetilde{T}(z)$ vanishes on $\partial\widehat{\calU}$ and, to the contrary, that there exists $f\in A^p$ and $g\in A^{p^\prime}$ such that 
\[
\limsup_{z\to\partial\widehat{\calU}} |\langle T_z f,g \rangle|>0.
\]
Thus there exists a sequence $\{z_k\}$ tending to $\partial\widehat{\calU}$ such that
\begin{equation}\label{eq:B(T)inC_0}
\lim_{k\to\infty} |\langle T_{z_k} f,g \rangle|>0.
\end{equation}
Since $\|T_{z_k}\|=\|T\|$ and both $A^p$ and $A^{p^\prime}$ are separable,
a diagonal argument on countable dense subsets, followed by reflexivity of $A^p$, yields a WOT-convergent subsequence. 
Passing to another subsequence if necessary, we may assume that
\[
T^\prime=\text{WOT}-\lim_{k\to\infty} T_{z_k}.
\]
Then for every $w\in\calU$, by Corollary \ref{cor:U_z^pk_w}, we have
\begin{align*}
\widetilde{T^\prime}(w) &= \lim_{k\to\infty} \langle T_{z_k} k_w^{(p)},k_w^{(p^\prime)} \rangle=\lim_{k\to\infty} \langle T U_{z_k}^p k_w^{(p)}, U_{z_k}^{p^\prime} k_w^{(p^\prime)} \rangle\\
&= \lim_{k\to\infty} \langle  T k_{\sigma_{z_k}^{-1}(w)}^{(p)}, k_{\sigma_{z_k}^{-1}(w)}^{(p^\prime)} \rangle
= \lim_{k\to\infty} \widetilde{T}(\sigma_{z_k}^{-1}(w)).
\end{align*}
This together with Lemma \ref{lem:sigma_ztoinfty} yields
\[
\widetilde{T^{\prime}}(w)=0
\]
for every $w\in\calU$. Since the Berezin transform is injective, it follows that $T^\prime=0$,
which contradicts \eqref{eq:B(T)inC_0}. This completes the proof of the lemma.
\end{proof}

Now, we are ready to prove Theorem \ref{thm:main}.

\begin{proof}[Proof of Theorem \ref{thm:main}]

(i) Let $T\in\mathcal{T}_p$. Then for any given $\varepsilon>0$, there exists $S$ which is a finite sum of finite products of Toeplitz operators with bounded symbols such that 
\[
\|T-S\|_{e, A^p \to A^p}\leq \|T-S\|<\varepsilon.
\]
Since $K_\bfi \in \mathcal{S}_{n+1}$, Theorem \ref{thm:essenstial} together with (b) of Lemma \ref{lem:supS_z} yields
\[
\|S\|_{e, A^p \to A^p} \simeq \sup_{\|f\|_p\leq 1} \limsup_{z\to\partial\widehat{\calU}} \|S_z f\|_p.
\]
Consequently, 
\[
\|T\|_{e, A^p \to A^p} \leq \|T-S\|_{e, A^p \to A^p} +\|S\|_{e, A^p \to A^p} \lesssim \varepsilon + \sup_{\|f\|_p\leq 1} \limsup_{z\to\partial\widehat{\calU}} \|S_z f\|_p.
\]
On the other hand, for every $f\in A^p$ with $\|f\|_p\leq 1$, 
\[
\|S_z f\|_p\leq \|(T-S)_z f\|_p +\|T_z f\|_p <\varepsilon + \|T_z f\|_p.
\]
A combination of these estimates gives 
\[
\|T\|_{e, A^p \to A^p} \lesssim \sup_{\|f\|_p\leq 1} \limsup_{z\to\partial\widehat{\calU}} \|T_z f\|_p+\varepsilon.
\] 
The reverse was already obtained in the first paragraph of the proof of Theorem \ref{thm:essenstial}.

(ii) The necessity follows immediately from Lemma \ref{lem:k_zweak}. 
For sufficiency, in view of (i), it is enough to show that
\[
\lim_{z\to \partial\widehat{\calU}} \|T_z f\|_p =0 
\]
for every $f\in A^p$.
By the density of $\mathcal{S}_{n+1}$ in $A^p$ and the fact that $\|T_z\|=\|T\|$, it suffices to prove the vanishing for $f\in\mathcal{S}_{n+1}$.

Assume now that $\widetilde{T}(z)$ vanishes on $\partial\widehat{\calU}$ and fix $f\in \mathcal{S}_{n+1}$ with $\|f\|_p\leq 1$. 
Write
\[
\|T_z f\|_p^p = \int_{D(\bfi,r)} |T_z f(w)|^p \,dV(w) + \int_{D(\bfi,r)^c} |T_z f(w)|^p \,dV(w),
\]
where $r$ will be chosen later.
By Lemma \ref{lem:B(T)inC_0}, we have $T_z f \to 0$ weakly in $A^p$ as  $z\to \partial\widehat{\calU}$.
Consequently, Lemma \ref{lem:weak,convergence} implies that  $T_z f \to 0$ uniformly on each compact subset of $\calU$. 
Hence, the first integral converges to 0 as $z\to \partial\widehat{\calU}$.

For any given $\varepsilon>0$, pick a finite sum of finite products of Toeplitz operators with bounded symbols $S$ such that $\|T-S\|<\varepsilon^{1/p}$.
Since $S$ satisfies (a) of Lemma \ref{lem:supS_z}, we can choose $r$ large enough such that 
\[
\int_{D(\bfi,r)^c} |S_z f(w)|^p \,dV(w) \leq \int_{D(\bfi,r)^c} \sup_{z\in\calU} |S_z f(w)|^p \,dV(w) <\varepsilon
\]
for all $z\in\calU$. Thus, 
\begin{align*}
&\int_{D(\bfi,r)^c} |T_z f(w)|^p \,dV(w) \\
&\leq 2^{p-1} \left(\int_{D(\bfi,r)^c} |(T-S)_z f(w)|^p \,dV(w)+\int_{D(\bfi,r)^c} |S_z f(w)|^p \,dV(w) \right) \\
&\leq 2^{p-1} \left( \|(T-S)_z f \|_p^p+\varepsilon\right) \leq 2^{p-1} \left(\|T-S\|^p+\varepsilon \right)<2^p \varepsilon
\end{align*}
for all $z\in\calU$.
Since $\varepsilon$ is arbitrary, we conclude that
\[
\lim_{z\to \partial\widehat{\calU}} \|T_z f\|_p =0 
\]
for every $f\in\mathcal{S}_{n+1}$, completing the proof of the theorem.
\end{proof}

\end{document}